\documentclass[a4paper,12pt]{amsart}

\usepackage{amsthm,amssymb,amsmath,amscd}
\usepackage[mathscr]{euscript}
\usepackage{textcomp}

\usepackage{graphicx}
\usepackage{tikz}

\usepackage[latin1]{inputenc}
\usepackage{ucs}

\theoremstyle{definition}
\newtheorem{definition}{Definition}

\theoremstyle{plain}
\newtheorem{theorem}[definition]{Theorem}
\newtheorem{lemma}[definition]{Lemma}

\newtheorem{observation}[definition]{Observation}

\newtheorem{remark}[definition]{Remark}

\PrerenderUnicode{\unichar{355}}

\author{Sylwia Antoniuk}

\address{Adam Mickiewicz University,
Faculty of Mathematics and Computer Science \textnormal{(on leave)},
ul.~Umultowska 87,
61-614 Pozna\'n, Poland}

\address{Instytut Matematyczny, Polska Akademia Nauk, ul.~\'Sniadeckich 8, 
00-656 Warszawa, Poland}

\email{\tt antoniuk@amu.edu.pl}

\author{Codru\unichar{355} Grosu}

\address{Freie Universit\"at, Berlin, Germany}

\thanks{CG and LN were supported by the Deutsche Forschungsgemeinschaft within the research training group `Methods for Discrete Structures' (GRK 1408).}

\email{grosu.codrut@gmail.com}

\author{Lothar Narins}

\email{\tt narins@math.fu-berlin.de}

\title[]
{On the connectivity Waiter-Client game}

\date{}

\begin{document}

\begin{abstract}
In this short note we consider a variation of the connectivity Waiter-Client game $WC(n,q,\mathcal{A})$ played 
on an $n$-vertex graph $G$ which consists of $q+1$ disjoint spanning trees. In this game in each 
round Waiter offers Client $q+1$ edges of $G$ which have not yet been offered. Client chooses one edge 
and the remaining $q$ edges are discarded. The aim of Waiter is to force Client to build 
a connected graph. If this happens Waiter wins. Otherwise Client is the winner. We consider 
the case where $2 < q+1 < \lfloor \frac{n-1}{2}\rfloor$ and show that for each such $q$ there exists a graph $G$ 
for which Client has a winning strategy. This result stands in opposition to the case 
where $G$ consists of just 2 spanning trees or where $G$ is a complete graph, since it has been shown 
that for such graphs Waiter can always force Client to build a connected graph.
\end{abstract}

\maketitle

\section{Introduction}

Waiter-Client games were first defined and studied by Beck (see e.g \cite{B2002}) 
under the name of Picker-Chooser games. These are positional games closely related 
to the well-studied Maker-Breaker games and Avoider-Enforcer games. Waiter-Client game 
$WC(n,q,\mathcal{A})$ is a two player, perfect information game played on the complete graph 
$K_n$, which proceeds in rounds. In each round, the first player, called Waiter, offers 
$q+1$ edges of $G$ which have not yet been offered. The second player, called Client, 
chooses one edge, and the remaining edges are discarded. The aim of Waiter is to 
force Client to build a graph that satisfies a given monotone property $\mathcal{A}$. 
If this happens Waiter wins. Otherwise Client wins. 

We consider the following version of Waiter-Client game $WC(n,q,\mathcal{A})$ which we call 
the connectivity Waiter-Client game. This time the game is played 
on an $n$-vertex graph $G$ which is the union of $q+1$ disjoint spanning trees 
and the aim of Waiter is to force Client to build a connected graph, i.e. $\mathcal{A}$ 
is the property of being connected.

Csernenszky et al. \cite{CMP2009} showed that for $q+1=2$ Waiter always has a winning strategy.  
Bednarska-Bzdega et al. \cite{BBHKL2014} showed that the same is true for $G$ 
being a complete graph, in which case $n$ is 
necessarily even. This result follows from the more 
general Theorem 3.3 in \cite{BBHKL2014} which says that in the 
Waiter-Client connectivity game $WC(n,q,\mathcal{A})$ played on the complete graph $K_n$, 
Waiter can always force Client to build a graph of size at least $\min\{n,2(n-q-1)\}$, 
provided that $n$ is sufficiently large. In particular, if $n$ is even and $q+1=n/2$ then
Waiter wins.

In \cite{BBHKL2014}, the authors posed a question whether in the connectivity Waiter-Client 
game $WC(n,q,\mathcal{A})$ played on a graph $G$ which is a disjoint union of 
$q+1$ spanning trees Waiter always has a winning strategy. We show that the answer to 
this question is negative. 

\begin{theorem}
\label{th1}
Let $2 < q+1 < \lfloor \frac{n-1}{2} \rfloor$. Then there exists an $n$-vertex graph $G$ 
which is a union of $q + 1$ disjoint 
spanning trees and such that in the connectivity Waiter-Client game $WC(n,q,\mathcal{A})$ 
played on $G$, Client has a winning strategy.
\end{theorem}

\begin{remark}
This leaves open the cases $q+1 = (n-1)/2$ for $n$ odd, and $q+1 = n/2 - 1$ for $n$ even. It may be tempting to believe that in these cases Client wins as well. However, we found that for $n=7$ and $n=9$, Waiter has a winning strategy.
\end{remark}

\section{The proof}

We first fix some notation. Whenever the game is played on a graph $G$, we proceed in rounds. 
In each round we delete from $G$ all edges offered by Waiter in this round 
and we denote by $G_i$ the graph obtained in this way, where $i$ is the number of the round.
In particular, $G_0 = G$ and $G_{n-1} = \emptyset$. The edges 
which have not yet been offered by Waiter are called \textit{free edges}. Moreover, we let $H_i$ denote 
the graph built on the same vertex set as $G$ and consisting of 
all the edges chosen by Client up till the $i$-th round. In particular, $H_0=\emptyset$ 
and $H = H_{n-1}$ is the graph built by Client when the game has finished. 
The Client wins if $H_{n-1}$ is not connected.

We make the following easy observations which hold for all rounds~$i$.
\begin{observation}
\label{obs:obs1}
If in the $i$-th round, there is a connected component $C \neq H_i$ in $H_i$ 
that has at most $q+2$ outgoing edges in $G_i$, 
then Waiter has the following options:
\begin{itemize}
 \item he offers only edges not incident to $C$;
 \item all the edges he offers are incident to $C$;
 \item in case the number of edges incident to $C$ in $G_i$ is exactly $q+2$, he offers 
 one of them and $q$ other edges not incident to $C$.
\end{itemize}
\end{observation}
Indeed, if Waiter offers $2 \leq i < q+1$ edges incident to $C$, then Client can refuse all of them. Then at any
later moment of time, Waiter can offer at most $q+2 - i < q+1$ edges incident to $C$. 
Consequently, Client can always refuse all of them,
leading to $C$ becoming an isolated component in $H_{n-1}$.
\begin{observation}
\label{obs:obs2}
If $C \neq H_i$ is a connected component in $H_i$ that has at most $q$ outgoing edges in $G_i$ then Client wins the game.
\end{observation}
\begin{observation}
\label{obs:obs3}
If $u$ and $v$ are isolated in $H_i$, $u$ and $v$ have degree $q+1$ in $G_i$ and are adjacent, then Client wins the game.
\end{observation}
Indeed, suppose without loss of generality that Waiter offers all $q+1$ edges incident to $u$.
Then Client chooses the edge $uv$. At any later moment of time, Client can discard any offered edge
that is incident to $v$, leading to $uv$ becoming an isolated component in $H_{n-1}$.

Before we proceed to the proof of Theorem~\ref{th1}, we need the following auxiliary lemma.

\begin{lemma}
\label{lemma1}
Let $K_{3k}$ be the complete graph on $3k$ vertices, where $k \geq 2$. Then there exist 
$k+1$ disjoint spanning trees in $K_{3k}$.
\end{lemma}
It follows from Nash-Williams theorem \cite{NWilliams} that $K_n$ has $\left\lfloor n/2 \right\rfloor$
edge-disjoint spanning trees. This in particular proves Lemma \ref{lemma1}.

We now show the following.
\begin{lemma}
\label{lemma2}
For any $q\geq 2$ there exists a graph $G=G(q)$ with 
$3(q+1)$ vertices, which is a disjoint union of $q+1$ spanning trees and such that 
Client has a winning strategy on $G$.
\end{lemma}

Note that this proves Theorem \ref{th1} in the case $n=3q+3$. We will later extend this construction 
to all $q$ satisfying $2 < q+1 < \lfloor \frac{n-1}{2} \rfloor$.
\begin{proof}[Proof of Lemma \ref{lemma2}]
The construction goes as follows.

We divide the vertices of $G$ into three sets $U=\{u_0, u_1, \ldots, u_q\}$, 
$V=\{v_0, v_1, \ldots, v_q\}$ and $W=\{ w_0, w_1, \ldots, w_q\}$. 
Next, we construct $q+1$ spanning trees $T_i$, $i=0,\ldots,q$.  
We first put the edges $\{u_0u_1, u_0v_0, v_0w_0\}$ into $T_0$, 
$\{v_0v_1, u_0w_0, w_0w_1\}$ into $T_1$, and $\{u_0u_i,v_0v_i,w_0w_i\}$ into $T_i$, 
where $i=2,\ldots, q$ (see Figure \ref{fig:fig1}). By Lemma~\ref{lemma1}, we can find 
$q+1$ disjoint spanning trees $T_0', T_1',\ldots ,T_q'$ on the vertices 
$(U \cup V \cup W) \setminus \{u_0, v_0, w_0\}$. We put the edges 
of $T_i'$ into $T_i$, for $i=0,\ldots,q$. Hence, we get $q+1$ disjoint 
spanning trees on the vertex set $U\cup V\cup W$. This is our graph $G=G(q)$.

\begin{figure}[!h]
\caption{The construction of $G(q)$}
\centering
\begin{tikzpicture}
\label{fig:fig1}

\filldraw (9*10:1.5cm) circle (1.5pt);
\filldraw (21*10:1.5cm) circle (1.5pt);
\filldraw (33*10:1.5cm) circle (1.5pt);

\filldraw (11*10:3.5cm) circle (1.5pt);
\filldraw (10*10:3.5cm) circle (1.5pt);
\filldraw (7*10:3.5cm) circle (1.5pt);
\filldraw (83:3.5cm) circle (0.5pt);
\filldraw (85:3.5cm) circle (0.5pt);
\filldraw (87:3.5cm) circle (0.5pt);

\filldraw (23*10:3.5cm) circle (1.5pt);
\filldraw (22*10:3.5cm) circle (1.5pt);
\filldraw (19*10:3.5cm) circle (1.5pt);
\filldraw (203:3.5cm) circle (0.5pt);
\filldraw (205:3.5cm) circle (0.5pt);
\filldraw (207:3.5cm) circle (0.5pt);

\filldraw (35*10:3.5cm) circle (1.5pt);
\filldraw (34*10:3.5cm) circle (1.5pt);
\filldraw (31*10:3.5cm) circle (1.5pt);
\filldraw (323:3.5cm) circle (0.5pt);
\filldraw (325:3.5cm) circle (0.5pt);
\filldraw (327:3.5cm) circle (0.5pt);

\draw (9*10:1.5cm) -- (21*10:1.5cm) -- (33*10:1.5cm) -- (9*10:1.5cm);

\draw (9*10:1.5cm) -- (11*10:3.5cm);
\draw (9*10:1.5cm) -- (10*10:3.5cm);
\draw (9*10:1.5cm) -- (7*10:3.5cm);

\draw (21*10:1.5cm) -- (23*10:3.5cm);
\draw (21*10:1.5cm) -- (22*10:3.5cm);
\draw (21*10:1.5cm) -- (19*10:3.5cm);

\draw (33*10:1.5cm) -- (35*10:3.5cm);
\draw (33*10:1.5cm) -- (34*10:3.5cm);
\draw (33*10:1.5cm) -- (31*10:3.5cm);

\node [right] at (9*10:1.5cm) {$u_0$};
\node [above] at (335:1.7cm) {$v_0$};
\node [above] at (205:1.7cm) {$w_0$};

\node [above] at (115:3.7cm) {$u_1$};
\node [above] at (103:3.7cm) {$u_2$};
\node [above] at (67:3.7cm) {$u_q$};

\node [right] at (350:3.7cm) {$v_1$};
\node [right] at (339:3.7cm) {$v_2$};
\node [right] at (307:3.8cm) {$v_q$};
 
\node [left] at (235:3.8cm) {$w_1$};
\node [left] at (221:3.8cm) {$w_2$};
\node [left] at (190:3.7cm) {$w_q$}; 
 
\end{tikzpicture}
\end{figure}
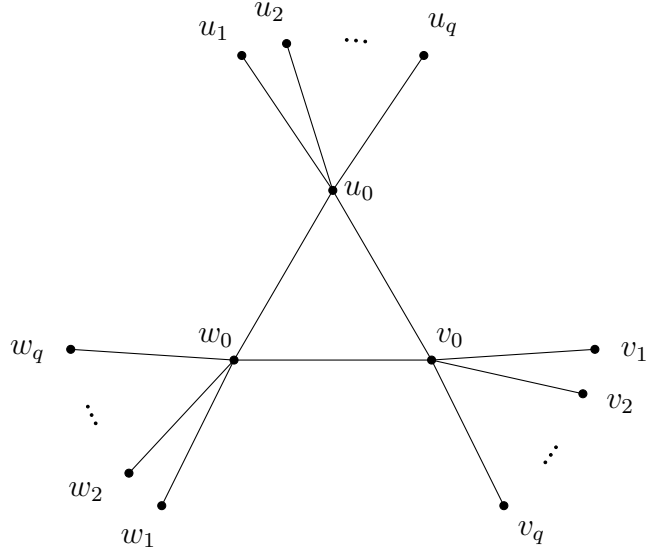


We present a strategy for Client which ensures that he we will win the game. 
Let $G'$ be the subgraph of $G$ consisting of all edges 
adjacent to at least one of the vertices $u_0$, $v_0$ or $w_0$ (this is precisely the graph in Figure \ref{fig:fig1}). 
We show that Client can isolate a subgraph of the triangle $u_0v_0w_0$, that is 
Client can ensure that in $H$ one of the vertices $u_0, v_0, w_0$ or one 
of the edges $u_0v_0$, $v_0w_0$, $w_0u_0$, 
or the whole triangle $u_0v_0w_0$ is a connected component.   

Now notice that the vertices $u_0$, $v_0$ and $w_0$ have degree $q+2$ in $G$. 
Therefore, by Observation \ref{obs:obs1}, the first time Waiter presents an edge incident to any of the above vertices, 
say to $u_0$, he must present either exactly one edge incident to $u_0$ or $q+1$ edges incident to that vertex. 
By symmetry the same argument works for $v_0$ and $w_0$.

Moreover, whenever Waiter presents an edge outside of $G'$, Client can choose this 
edge and thus the number of free edges in $G'$ can only drop down. It is easy to see 
that if Client has a winning strategy on the graph $G'$ then the same strategy works 
for the graph $G'$ with some of the edges deleted. Thus we may focus on the game 
played entirely on the graph $G'$ (and the same parameter $q$).

We shall assume by symmetry that
in the first round one of the edges offered is adjacent to $u_0$. We have two subcases.

\textit{Case 1: All of the edges offered are incident to $u_0$.}

Then either Waiter presents both $u_0v_0$ and $u_0w_0$, or just one 
of them, say $u_0v_0$.

In the first case, Client chooses a third edge 
(not adjacent to $v_0$ nor $w_0$) and by Observation \ref{obs:obs3} he wins the game.

In the second case, Client chooses $u_0v_0$. There are only $q+2$ free edges incident to $u_0v_0$ now.
If in the next round, both edges $u_0w_0$ and $v_0w_0$ are offered, then all the other edges 
offered must be incident to 
$w_0$ too. Client chooses one of these edges, leaving only $q$ free edges incident to $u_0v_0$, 
and thus winning the game.
If in the second round, exactly one of $u_0w_0$ and $v_0w_0$ 
is offered, then Client chooses this edge. There are only $q$ free edges left incident to 
the triangle $u_0v_0w_0$, and so again Client wins.
So in the second round, by Observation \ref{obs:obs1}, Waiter offers $q$ edges incident to 
$v_0$ and $1$ edge incident to $w_0$. 
Client then picks the edge incident to
$w_0$, leaving only $2 \leq q$ free edges incident to $u_0v_0$. So Client wins in this last situation as well.

\textit{Case 2: Exactly one of the edges offered is incident to $u_0$.}

We may assume that Case 1 does not apply to $v_0$ nor $w_0$, otherwise we are done. 
Then by Observation \ref{obs:obs1}, Waiter offers one edge incident to
$u_0$, one incident to $v_0$, and one incident to $w_0$. Thus $q=2$. Furthermore the edge incident to $u_0$
is either $u_0u_1$ or $u_0u_2$. Client then chooses this edge.

In the second round Waiter has to offer at least one edge incident to $\{v_0,w_0\}$. Consequently
by Observation \ref{obs:obs1}, Waiter must offer three edges adjacent to 
one of the vertices $v_0$ or $w_0$. Then Client chooses the edge $v_0w_0$. There are only $2$ edges left incident with $v_0w_0$, and
so Client wins by Observation \ref{obs:obs2}.

\end{proof}

\begin{proof}[Proof of Theorem \ref{th1}]

The crucial part of the above argument is the existence of a subgraph $G'$ in which we have 
three vertices spanning a triangle and such that the degree of each of them in $G$ 
is equal precisely to $q+2$. By repeating the same reasoning, one can ensure 
oneself that in the connectivity game played on the graph $G$ of this form, that 
is having such a subgraph $G'$, Client always has a winning strategy. In the remaining 
part of the proof we show that for any $3 \leq q+1 < \lfloor \frac{n-1}{2}\rfloor$ 
there exists such a graph $G = G(n,q)$.

Let $n \geq 9$ and $3 \leq q+1 < \left\lfloor \frac{n-1}{2}\right\rfloor$. 
We use the induction to construct $G(n,q)$ from $G(n-2,q-1)$. 
To start the induction, we need a graph $G(n - 2(q-2),2)$
on $n - 2(q-2) \geq 9$ vertices, which is the union of three spanning trees 
and which has three fixed vertices $u_0$, $v_0$, $w_0$ each of degree four 
and such that they span a triangle. For example, we can take the graph $G(2)$ 
which has 9 vertices, add to it $n - 2(q-2) - 9$ new vertices and add three edges 
between each new vertex and three arbitrary vertices from $G(2)$ different from 
$u_0$, $v_0$, $w_0$.

The construction goes as follows. Assume that we have already constructed the graph $G(n_0,q_0)$ 
on $n_0$ vertices which is the union of $q_0+1$ spanning trees and  
in which $u_0$, $v_0$, $w_0$ all have degree $q_0+2$. We now split vertices of $G(n_0,q_0)$ 
into three sets $V_1$, $V_2$ and $V_3$. The set $V_1$ consists of the vertices $u_0$, $v_0$, $w_0$, 
and we split the remaining vertices into two sets of size $\lfloor \frac{n_0-3}{2} \rfloor$ 
and $\lceil \frac{n_0-3}{2} \rceil$ respectively. Let $T_0,T_1,\ldots,T_{q_0}$ denote 
the spanning trees of $G(n_0,q_0)$. We add two new vertices $u$ and $v$. 
We then use the edges between $u$ and $V_2$ to extend 
each spanning tree $T_i$ to a spanning tree $T_i'$ of $G(n_0,q_0) \cup \{u\}$. Note that this is possible,
as $\lfloor \frac{n_0-3}{2} \rfloor = \left\lfloor \frac{n_0-1}{2}\right\rfloor - 1 \geq q_0+1$.
Similarly, we use the edges between $v$ and $V_3$ to extended 
each $T_i'$ to a spanning tree of $G(n_0,q_0) \cup \{u, v\}$. Finally, we construct a 
new spanning tree $T$ of $G(n_0,q_0) \cup \{u, v\}$ from the edges $uv$, $uu_0$,
$uv_0$, $uw_0$ and all edges between $u$ and $V_3$ and all edges between $v$ and $V_2$.
\end{proof}

\noindent
\textit{Acknowledgements}. We would like to thank Tibor Szab\'o for bringing the problem studied in this paper to our attention.

\bibliographystyle{plain}

\begin{thebibliography}{}

\bibitem{B2002} J.~Beck, {\em Positional games and the second moment method}, 
Combinatorica 22 (2002), 169--216.
 
\bibitem{BBHKL2014} M.~Bednarska-Bzdega, D.~Hefetz, M.~Krivelevich, T.~\L uczak, 
{\em Manipulative waiters with probabilistic intuition}, preprint.
 
\bibitem{CMP2009} A.~Csernenszky, C.~I.~M\'{a}ndity, A.~Pluh\'{a}r, 
{\em On Chooser-Picker positional games}, Discrete Mathematics 309 (2009), 5141--5146.

\bibitem{NWilliams} C.S.J.A.~Nash-Williams, {\em Edge-Disjoint Spanning Trees of Finite Graphs}, J. London Math. Soc. (1961) s1-36 (1): 445-450.
 
\end{thebibliography}

\end{document}